\documentclass{amsart}
\usepackage[utf8]{inputenc}
\usepackage[svgnames]{xcolor}
\usepackage{amsbsy}
\usepackage{amssymb}
\usepackage{graphicx}
\usepackage[backend=bibtex]{biblatex}
\usepackage{hyperref}
\hypersetup{colorlinks=true, citecolor=green!60!black}
\usepackage[capitalise]{cleveref}

\usepackage{tikz}

\newtheorem{thm}{Theorem}[section]
\Crefname{thm}{Theorem}{Theorems}
\newtheorem{lem}[thm]{Lemma}
\Crefname{lem}{Lemma}{Lemmas}
\newtheorem{prop}[thm]{Proposition}
\Crefname{prop}{Proposition}{Propositions}
\newtheorem{cor}[thm]{Corollary}
\Crefname{cor}{Corollary}{Corollaries} 

\Crefname{que}{Question}{Questions}

\Crefname{con}{Conjecture}{Conjectures}

\Crefname{clm}{Claim}{Claims}

\Crefname{goal}{Goal}{Goals}

\theoremstyle{definition}
\newtheorem{defn}[thm]{Definition}
\Crefname{defn}{Definition}{Definitions}
\newtheorem{ex}[thm]{Example}
\Crefname{ex}{Example}{Examples}
\newtheorem{rem}[thm]{Remark}
\Crefname{rem}{Remark}{Remarks}
\bibliography{bibliography}

\tikzset{separrow/.pic = {
		\draw [thick,->] (-.2cm,-.15cm) -- (.2cm, -.15cm);
		\node at (0cm, -.4cm) {#1};},
	separrow inv/.pic = {
		\draw [thick,<-] (-.2cm,-.15cm) -- (.2cm, -.15cm);
		\node at (0cm, -.35cm) {#1};},
}

\DeclareMathOperator{\im}{im}

\newcommand{\Pcal}{\ensuremath{\mathcal{P}}}
\newcommand{\QP}{\ensuremath{\mathcal{Q}_P}}

\newcommand{\inv}{\mathord{^*}}

\colorlet{picturegreen}{green!60!black}

\begin{document}
	\title{Limit-closed Profiles}
	\author{Ann-Kathrin Elm}
	\address{Research conducted at Universität Hamburg, Department of Mathematics, Hamburg, Germany}
	\curraddr{Ann-Kathrin Elm: Universität Heidelberg, Department of Computer Science, Heidelberg, Germany}
	\email{elm@informatik.uni-heidelberg.de, hendrikheine@gmail.com}
	\author{Hendrik Heine}
	\subjclass[2020]{06-XX (Primary), 05C63, 05C05 (Secondary)}
	\keywords{infinite abstract separation system, tree set, tangle-tree theorem, profile}
	
\begin{abstract}
	Tangle-tree theorems are an important tool in structural graph theory, and abstract separation systems are a very general setting in which tangle-tree theorems can still be formulated and proven.
	For infinite abstract separation systems, so far tangle-tree theorems have only been shown for special cases of separation systems, in particular when the separation system arises from a (locally finite) infinite graph.
	We present a tangle-tree theorem for infinite separation systems where we do not place restrictions on the separation system itself but on the tangles to be arranged in a tree.
\end{abstract}
\maketitle

\section{Introduction}

Tangles were first introduced by Robertson and Seymour in \cite{GraphminorsX} as an obstruction to high tree-width. One of the important ingredients in their graph minor project was a tangle-tree theorem, that is a theorem giving a tree decomposition separating all the tangles of a graph.

As it turns out these features of tangles, both being an obstruction and admitting a tangle-tree theorem, can be formulated more abstractly.
This is the core of the theory of abstract separation systems formulated in \cite{ASS}.
In this setting tangle-tree theorems are now usually formulated more generally in terms of profiles instead of tangles and tree sets instead of tree decompositions, as seen for instance in \cite{ProfilesNew}.
While that article gives a very general such theorem for finite separation systems, for infinite separation systems there is no general tangle-tree theorem.
But there are several tangle-tree theorems for special cases of infinite separation systems:
for separation systems that come from a (locally finite) graph (see for example \cite{CanonicalTreesofTdc}), are the inverse limit of finite separation systems \cite{ProfiniteSepsys}, or in which every separation crosses only finitely many others \cite{ToTInfSepsys}.
In this article we consider a property not of the infinite separation system but of the profiles to be distinguished in order to obtain a tangle-tree theorem.
In particular we ask that the profiles to be distinguished are closed under taking limits (the formal definition follows in \cref{sec:equivclasses}).

In that same section we find, for a separation system and a set of regular closed profiles, a tree set distinguishing the profiles in a two step process.
Here, a regular profile is a profile that does not contain certain separations that behave very counter-intuitively.
First, we consider a set of equivalence classes of separations, show that these form a separation systems with special properties, and find a tree set of such equivalence classes.
Then we choose representatives of for these equivalence classes that form a tree set.
In \cref{sec:nonregular} we show that the regularity can be dropped from the requirements on the profiles by slightly adjusting the separation system such that it does not contain separations with counter-intuitive behaviour.
\Cref{sec:ttt} then applies the result from \cref{sec:equivclasses} inductively to get a full tree of tangles for a separation system with an order function.
The main point here is that the order function defines a chain of subsystems of the separation system, each with its own profiles, and the resulting tree of tangles distinguishes all profiles of the distinct subsystems at once.

\section{Preliminaries}

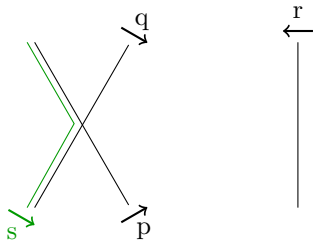
\begin{figure}
	\centering
	\begin{tikzpicture}
		\draw (-.5, 1.1) -- ++ (-60:2.5cm) pic [rotate = 30] {separrow = p};
		\draw (-.5, -1.1) -- ++(60:2.5cm) pic [rotate = 150] {separrow inv= q};
		\draw (3, -1.1) -- ++(90:2.2) pic [rotate=180] {separrow = r};
		\draw [picturegreen] (-.6, 1.1) -- ++(-60:1.25 cm) -- (-.6, -1.1) pic [rotate = -30] {separrow = s};
	\end{tikzpicture}
	\caption[A visualisation of nested and crossing separations, and of an infimum of two separations]{In this figure, separations are depicted as a line with a small arrow attached. This visualisation comes from the example of a separation system where the separations are subsets of some ground set $E$ and the partial order and involution are the subset relation and complement operation. For a separation $s$, i.e. a subset of $E$, the line then depicts the border between $s$ and its complement, and the arrow points towards the side which is the complement of $s$.\\
	So in this illustration, $p$ and $q$ cross and $p$ points towards $r$. Also $s$ is the infimum of $p$ and $q$ and thus would technically have to be depicted slightly more to the right, but then the illustration would be much harder to understand.
	This type of visualisation does not work for trivial and degenerate separations, as they cannot be properly represented as subsets of a ground set.}
	\label{fig:basicseps}
\end{figure}

We will assume the reader to be familiar with the basic definitions and facts about separation systems and (submodular) universes, which are explained in detail in \cite{ASS}.
In this paper, all order functions will only have values in the integers.
Also, we will only need the notion of oriented separations and not the one of unoriented separations.
So we will shorten the term ``oriented separation'' to just ``separation''.
In a submodular universe $(U,\leq, \inv)$, we will denote separations of order exactly $k$ as \emph{$k$-separations}, separations of order at most $k$ as \emph{$\leq k$-separations} and separations of order less than $k$ as \emph{$<k$-separations}.
As defined in \cite{SeparationsOfSets}, a homomorphism of separation systems $(S,\leq, \inv)$ and $(S',\leq', {\inv}')$ is a map $f:S\rightarrow S'$ such that for all separations $s$ and $t$ in $S$, $f(s^*)={f(s)^*}'$ and $s\leq t \Rightarrow f(s)\leq' f(t)$.

The following definitions of profiles can be found in \cite{ProfilesNew}, and the structurally submodular separation systems are defined for example in \cite{structurallysubmodular}.
Let $(U,\leq, \inv)$ be a universe of separations and $(S,\leq, \inv)$ a subsystem.
In this context, infima and suprema of elements of $S$ are always taken in the universe\footnote{The alternative would be to take them in the restriction of $\leq$ to $S$, where suprema and infima need not always exist.} and thus always exist, but need not be elements of $S$.
$S$ is \emph{structurally submodular} if for all separations $s$ and $t$ in $S$ at least one of $s\vee t$ and $s\wedge t$ is contained in $S$.
A \emph{profile of $S$} is a consistent orientation of $S$ such that for any elements $s$ and $t$ of the profile the separation $(s\vee t)^*$ is not contained in the profile.
In particular, if $s\vee t$ is contained in $S$ then it is also contained in the profile, and a separation system containing a degenerate separation does not have a profile.
Note that the definition of a profile thus not only depends on the separation system but also on the surrounding universe.
If $U$ is submodular, then a \emph{$k$-profile} of $U$ is a profile of $S_k$, the subsystem of $U$ consisting of the separations of order less than $k$.
A \emph{profile of $U$} is a $k$-profile of $U$ for some $k\in \mathbb{N}$, and $k$ is its \emph{order}.

Given a $k$-profile $P$ of $U$ and $l\in \mathbb{N}$ with $l\leq k$, the set $P\cap S_l$ is the $l$-profile \emph{induced} by $P$.
Two profiles $P$ and $Q$ of $U$ are \emph{distinguished} by a separation $s$ if $s\neq s^*$ and one of $s$ and $s^*$ is contained in $P$ and the other one in $Q$.
If $U$ is submodular and the order of $s$ is minimal among all separations distinguishing $P$ and $Q$, then $s$ distinguishes them \emph{efficiently}.
A set of separations $T$ distinguishes a set $\mathcal{P}$ of profiles of $U$ (efficiently) if for every pair of distinguishable profiles in $\mathcal{P}$ (that is, profiles for which there is a separation distinguishing them) there is a separation in $t$ distinguishing them (efficiently).
A profile $P$ of $U$ is \emph{robust} if for every $r\in P$ and every $l$-separation $s$ the following holds: If the orders of both $r^*\wedge s$ and $r^*\wedge s^*$ are less than the order of $r$, then $P$ does not contain both $r^*\wedge s$ and $r^*\wedge s^*$.
There are also other, slightly weaker definitions of robustness.
But that is not relevant, as robustness is never used directly, only the following consequence \cite[Lemma 3.6]{ProfilesNew}:
In a submodular universe, if $r$ efficiently distinguishes two robust profiles $P$ and $P'$ and $s$ efficiently distinguishes two robust profiles $Q$ and $Q'$ and the order of $r$ is less than the order of $s$, then one of $r\wedge s$, $r\wedge s^*$, $r^*\wedge s$ and $r^*\wedge s^*$ efficiently distinguishes $Q$ and $Q'$.

\section{Equivalence classes}\label{sec:equivclasses}

Let $S$ be a separation system in some universe $U$ and $\mathcal{P}$ a set of regular profiles of $S$. Then we can define a map $f_{\mathcal{P}}: S \rightarrow 2^{\mathcal{P}}$ via $f_{\mathcal{P}}(s)= \{P \in \mathcal{P}| s^* \in P\}$. Intuitively, the image of $f_{\mathcal{P}}$ condenses from $S$ the information how it distinguishes the elements of $\mathcal{P}$. We consider $2^\mathcal{P}$ as a separation system with inclusion as the order.

\begin{lem} \label{maphom}
	The map $f_{\mathcal{P}}$ is a homomorphism of separation systems which respects $\vee$ and $\wedge$.
\end{lem}
\begin{proof}
	Since profiles orient every separation, $f_{\mathcal{P}}$ respects the involution. Furthermore if $s \leq t$ and $s^* \in P$ for some profile $P$, then we cannot have $t \in P$ by consistency, so $f_{\mathcal{P}}(s) \leq f_{\mathcal{P}}(t)$.
	Now it suffices to prove that for $s,t \in S$ with $s \vee t \in S$ an arbitrary $P \in \mathcal{P}$ contains $s \vee t$ if and only if it contains both $s$ and $t$. The forward implication follows by consistency, the backwards one by the profile property.
\end{proof}

The fibers of $f_{\mathcal{P}}$ are exactly the equivalence classes obtained by regarding two separations as equivalent if they are oriented the same way by every profile in $\mathcal{P}$. Assuming a structurally submodular $S$, comparing these equivalence classes via their images under $f_\mathcal{P}$ gives the same partial order as comparing them via their elements.  In fact, a slightly weaker condition suffices. Call $S$ \emph{weakly $\mathcal{P}$-submodular (with respect to $\mathcal{P})$} if whenever $f_{\mathcal{P}}(s) \leq f_{\mathcal{P}}(t)$ at least one of $s \vee t$ and $s \wedge t$ is contained in $S$.
\begin{prop} \label{leqequiv}
	Let $S$ be weakly $\mathcal{P}$-submodular with respect to $\mathcal{P}$.
	Then for $A,B \in \im(f_{\mathcal{P}})$ we have $A \leq B$ if and only if there are $a \in f_{\mathcal{P}}^{-1}(A), b \in f_{\mathcal{P}}^{-1}(B)$ with $a \leq b$.
\end{prop}
\begin{proof}
	The backward direction is immediate by \cref{maphom}. For the forward direction choose $a \in f_{\mathcal{P}}^{-1}(A), b \in f_{\mathcal{P}}^{-1}(B)$. At least one of $a \wedge b$ and $a \vee b$ is contained in $S$ by weak submodularity.  If $a \wedge b \in S$, then by \cref{maphom} it is contained in  $f_{\mathcal{P}}^{-1}(B)$ and thus forms the required pair together with $b$. Similarly, if $a \vee b \in S$, then $a \vee b \in f_{\mathcal{P}}^{-1}(A)$ and it forms the required pair with $a$.
\end{proof}

We want to use the transformation $f_{\mathcal{P}}$ to find a distinguishing set for $\mathcal{P}$. We will proceed in two steps: First we will look for an abstract distinguishing set in the image and then look for separations of $S$ to represent them.

Let us start by stating our objective for the first step more formally. We will say that $A \in \im(f_{\mathcal{P}})$ \emph{separates} $P,Q \in \mathcal{P}$ if $P \in A$ and $Q \notin A$ or vice versa. Then we are looking for some tree set $T \subseteq \im(f_{\mathcal{P}})$ such that different elements of $\mathcal{P}$ are always separated by some element of $\im(f_{\mathcal{P}})$. Since structural submodularity of $S$ translates to $\im(f_{\mathcal{P}})$, for finite $\mathcal{P}$ standard techniques easily show that this condition is enough to reach our goal. If $\mathcal{P}$ is infinite these standard methods are not sufficient, but there is a useful idea, seen for instance in \cite{Mat2seps} or \cite{CunninghamEdmonds}, which may help. That idea is taking only the \emph{good} separations, that is those not crossed by any other separation, for our tree set.

In the following we will show that, given certain conditions, the set $T(S,\mathcal{P})$ of good separations of $\im(f_{\mathcal{P}})$ (except $\emptyset$ and $\mathcal{P}$) does indeed meet our demands. 
Since $T(S,\mathcal{P})$ is a tree set by definition, all that needs to be shown is that $T(S,\mathcal{P})$ separates any two distinct elements of $\mathcal{P}$. When dealing with finite separation systems, it is sometimes useful to consider maximal separations. If we want to use this trick in the infinite case, we encounter some difficulties. First of all, profiles may not even have maximal elements, which would render our strategy impossible. Thus we require our profiles to be \emph{closed}, meaning that any chain in the profile has a supremum in the universe which is contained in the profile. This ensures that each profile has a maximal element, but even these maximal elements may not have the nice properties which we are used to, say when we have an order function.

Thus we need one more condition, which emulates some of the additional structure provided by an order function. We call $S$ \emph{orderly} (with respect to $\mathcal{P}$) if for any $s,t \in S$ such that both $\{s,t\}$ and $\{s^*,t^*\}$ are subsets of (possibly different) elements of $\mathcal{P}$ we have $s \vee t \in S$ and $s \wedge t \in S$. 
To give an example, it will follow from \cref{orderly} that the set of proper $k$-separations in a $k$-connected graph is orderly.

\begin{lem}\label{separatedbygoodseps}
	If $S$ is orderly and every $P \in \mathcal{P}$ is closed, all different $P,Q \in \mathcal{P}$ are separated by some $t \in T(S,\mathcal{P})$.
\end{lem}
\begin{proof}
	Let $C$ be a chain in $S$ such that each element of $C$ is contained in $P$ and not in $Q$ which is maximal with these properties. Since $P$ is closed, $C$ has a supremum $s$, which is contained in $P$. Furthermore, by consistency, we have $s \notin Q$. Now it suffices to prove that $f_\mathcal{P}(s)$ is good. If not, there exists some $x \in S$ such that the images of $s$ and $x$ cross, without loss of generality $x \in P$. But since $S$ is orderly, this would imply $s \vee x \in S$ and hence $s \vee x \in P$ and this separation could have been added to $C$.
\end{proof}
Now let us consider the second step. Starting with a regular tree set $T$ in $\im(f_{\mathcal{P}})$, we want to find a nested set of preimages (or equivalently one isomorphic to $T$). Once again we want to use maximal separations and so need the same conditions as before. Closure guarantees that each equivalence class has a maximal separation and orderliness that it is even greatest.
\begin{lem}\label{biggestequivalent}
	If every $P \in \mathcal{P}$ is closed, for every $t \in T$ the set $f_\mathcal{P}^{-1}(t)$ has a greatest element.
\end{lem}
\begin{proof}
	Let $X = f_\mathcal{P}^{-1}(t)$ and let $C$ be a nonempty chain in $X$. Since $\mathcal{P} \setminus t$ is nonempty and each profile in that set is closed, $C$ has a supremum $s$ with $f_\mathcal{P}(s)\supseteq t$. Conversely there cannot be any $Q \in t$ with $s \in Q$, since we would then have $C \subseteq Q$ by consistency, a contradiction. Thus $s \in X$.
	This implies that any element of $X$ lies below some maximal element. 

	Now let $a$ and $b$ be two maximal elements of $X$. Since $S$ is orderly, $a \vee b \in S$ and then also $a \vee b \in X$. By maximality of $a$ and $b$ these must then both be equal to $a \vee b$ and we have $a = b$. Thus $X$ only has a single maximal element which must then be greater than all other separations in $X$.
\end{proof}
Let $m$ be the function mapping each $t \in T$ to the greatest element of $f_\mathcal{P}^{-1}(t)$.
We would like to use $m$ to choose the representatives, however, this is not quite possible, since for $t \in T$ the separations $m(t)$ and $m(t^*)$ are usually not inverses.
Fortunately, this is no great obstacle. If we simply choose one of these two possible unoriented separations, the only thing that could go wrong is that for $s \leq t$ we choose $m(s)$ and $m(t^*)^*$ as the representing separations. It thus suffices to choose exactly opposite to a consistent orientation, which always exists.
So we fix a consistent orientation $o$ of $T$ and define the function $m_o$ by mapping $s \in T$ to $m(s^*)^*$ for $s \in o$ and to $m(s)$ for $s \notin o$. Let $\hat{T}$ be the image of $T$ under $m_o$.
\begin{cor}
	Let $S$ be orderly and every $P \in \mathcal{P}$ closed. Then $f_\mathcal{P}$ restricts to an isomorphism between $\hat{T}$ and $T$.
\end{cor}
\begin{proof}
	Clearly, $m_o$ is the inverse map of $f_\mathcal{P}$ restricted to the image of $m_o$. Since $f_\mathcal{P}$ is a homomorphism by \cref{maphom}, it suffices to show that $m_o$ is, too. So let $s,t \in T$ be such that $s \leq t$. We need to show that $m_o(s) \leq m_o(t)$. If $t \notin o$, we have $m_o(t) = m(t)$. Since $S$ is orderly, we have $m(t) \vee m_o(s) \in S$. By \cref{maphom} $f_\mathcal{P}(m(t) \vee m_o(s)) = f_\mathcal{P}(m(t)) \cup f_\mathcal{P}(m_o(s)) = t \cup s=t$. But since $m(t)$ is a greatest element, we must have $m_o(s) \leq m(t) = m_o(t)$. So we may assume $t \in o$ and by consistency also $s \in o$. Then we have $m_o(s)=m(s^*)^*$ and $m_o(t)=m(t^*)^*$. Since $S$ is orderly, we have $m(s^*) \vee m(t^*) \in S$ and calculating with \cref{maphom} as before we get $f_\mathcal{P}(m(s^*) \vee m(t^*)) = s^*$. Again by choice of $m(s^*)$ we must have $m(s^*) \geq m(t^*)$ and thus $m_o(s) \leq m_o(t)$.
\end{proof}
This completes the second step. Overall, we have now proven the following theorem.
\begin{thm} \label{abstracttheorem}
	Let $S$ be a regular separation system orderly with respect to a set $\mathcal{P}$ of closed profiles. Then there is a tree set $T$ with the following properties:
	\begin{enumerate}
		\item{Any two different elements $\mathcal{P}$ are distinguished by some element of $T$.}
		\item{Any element of $T$ distinguishes some elements of $\mathcal{P}$.}
		\item{In the set $P \cap T$ for every $P \in \mathcal{P}$ every separation lies below some maximal separation.}
	\end{enumerate}
\end{thm}

\section{Non-regular profiles}\label{sec:nonregular}
In this section we want to show that in \cref{abstracttheorem}, the regularity condition can be dropped.
In order to do so, we will slightly adjust $S$ and $\mathcal{P}$ to a regular separation system with a modified set of profiles.
Then the modified profiles are necessarily regular, and theory for regular profiles can be applied to them and the adjusted separation system.
As a last step we will show that any tree set distinguishing the modified profiles also distinguishes $\mathcal{P}$.

There is a procedure to make a separation system regular in \cite{ASS}: first taking its \emph{essential core}, i.e.\ deleting all trivial, co-trivial and degenerate elements, and then taking the \emph{regularization} of that essential core, that is, dropping relations of the form $s\leq s^*$ from the partial order.
This two-step-procedure is necessary because dropping relations of the form $s\leq s^*$ from the partial order need not yield another partial order if there are degenerate or trivial elements present.
But if $S$ is a subsystem of a surrounding universe, it is in general not possible to just delete trivial elements from the universe, and then it is not possible to adjust the surrounding universe to also be a surrounding universe for the regularization.
In order to overcome this problem, we will have to relax the notion of corners in a separation system once more so we can cope without a surrounding universe.

\begin{defn}
	Let $(S, \mathord{^*}, \leq)$ be a separation system.
	A \emph{corner map} is a map $\vee$ from a subset of $S\times S$ to $S$ such that
	\begin{itemize}
		\item if $s\leq t$, then $\vee(s,t)$ is defined, and
		\item if $\vee(s,t)$ is defined, then also $\vee(t,s)$ is defined and it is the supremum of $s$ and $t$ in the partial order $\leq$.
	\end{itemize}
	For a separation system with corner map, in analogy to a separation system that is a subsystem of a universe, we say that a corner $s\vee t$ \emph{exists} in $S$ or \emph{is contained in $S$} if $\vee(s,t)$ is defined.
	Also, $\vee(s,t)$ will be denoted by $s\vee t$, and $s\wedge t$ is a shorthand for $(s^*\vee t	^*)^*$.
\end{defn}

\begin{ex}
	If $S$ is a separation system that is a subsystem of some universe $U$, then $S$ naturally comes with a corner map where $\vee(s,t)$ is defined if $s \vee t$, which exists in the surrounding universe, is contained in $S$.
	Another example of a corner map, which can be defined without taking a surrounding universe into account, is to define $\vee(s,t)$ to exist whenever $s$ and $t$ have a supremum in the partial order $\leq$.
	As every separation system can be embedded into a universe in a way that preserves suprema of $\leq$ (see \cite[Theorem 3.1]{EKTStructureSubmod}), this example is actually a special case of the previous example for a corner map.
	As a third example, one can define a corner map where $\vee(s,t)$ is defined only when $s\leq t$ or $t \leq s$.
\end{ex}

Formally, we have to redefine our terminology for separation systems with corner map.
But of course all these definitions will be the same as before and just as expected.

\begin{defn}
	Let $(S,\leq, \mathord{^*}, \vee)$ be a separation system with a corner map.
	A \emph{profile} of the separation system with corner map is a consistent orientation $P$ of $(S,\leq,\mathord{^*})$ such that for all $s,t$ in $P$, if $s\vee t$ is defined then $(s \vee t)^*$ is not contained in $P$.
	Given a set of profiles $\mathcal{P}$, the separation system with corner map is \emph{orderly} with respect to $\mathcal{P}$ if for all separations $s$ and $t$ such that some profile in $\mathcal{P}$ contains $s$ and $t$ and some profile in $\mathcal{P}$ contains $s^*$ and $t^*$ then both $s\vee t$ and $s^*\vee t^*$ are defined.
	A profile $P$ is \emph{closed} in $S$ if for all $\leq$-chains in $P$ the supremum with respect to $\leq$ exists in $S$ and is contained in $P$.
\end{defn}

Note that if $S$ is a separation system that is a subsystem of a universe $U$, then a consistent orientation of $S$ is a profile of $S \subseteq U$ if and only if it is a profile of $S$ with the induced corner map, and a profile of $S$ that is closed with respect to $S$ as a subsystem of $U$ is also closed in $S$ itself.
Furthermore, if $\mathcal{P}$ is a set of profiles of $S$, then $S \subseteq U$ is orderly with respect to $\mathcal{P}$ if and only if $S$ with the induced corner map is orderly with respect to $\mathcal{P}$.

Now we want to apply regularization to a separation system with corner map and a set of profiles.
Through this process we can essentially keep the corner map and the set of profiles, and we preserve orderliness and closedness.

\begin{defn}
	Let $(S,\leq, \mathord{^*}, \vee)$ be a separation system with a corner map.
	The \emph{essential core} of $(S,\leq, \mathord{^*}, \vee)$ is the set $S'$ of separations in $S$ that are neither degenerate nor trivial nor co-trivial, together with the restrictions of $\leq$, $\mathord{^*}$ and $\vee$ to $S'$.
	The \emph{regularization} of the essential core is the tuple $(S', \leq', \mathord{^*}', \vee')$ where $s\leq' t$ for elements $s$ and $t$ of $S'$ if and only if $s\leq t$ and $s\neq t^*$, $\mathord{^*}'$ is the restriction of $\mathord{^*}$ to $S'$ and $\vee'$ is defined for all those elements $s$ and $t$ of $S'$ for which $s\vee t$ is defined, contained in $S'$, and the supremum of $s$ and $t$ in $\leq'$.
\end{defn}

\begin{lem}\label{regularization}
	Let $\mathcal{P}$ be a set of profiles of a separation system with corner map $(S,\leq, \mathord{^*}, \vee)$.
	Then the regularization $(S',\leq',\mathord{^*}', \vee')$ of the essential core is a regular separation system with corner map of which $\mathcal{P}':=\{P\cap S' \colon P\in \mathcal{P}\}$ is a set of profiles.
	Moreover, if $S$ is orderly with respect to $\mathcal{P}$, then $S'$ is orderly with respect to $\mathcal{P}'$, and if the elements of $\mathcal{P}$ are closed with respect to $\leq$ then the elements of $\mathcal{P}'$ are closed with respect to $\leq'$.
\end{lem}
\begin{proof}
	It has already been shown in \cite{ASS} that $(S',\leq', \mathord{^*}')$ is a regular separation system, and $\vee'$ is clearly a corner map of that separation system.

	Let $P\in \mathcal{P}$.
	By the definition of profile of a separation system with corner map, $P\cap S'$ is a profile of $(S',\leq', \mathord{^*}', \vee')$.
	We will now show that if $P$ is closed with respect to $\leq$, then $P\cap S'$ is closed with respect to $\leq'$.
	So assume $P$ is closed with respect to $\leq$ and let $(s_i)_{i\in I}$ be a $\leq'$-chain of elements of $P\cap S'$.
	Then $(s_i)_{i\in I}$ is also a $\leq$-chain of elements of $P$, so it has a supremum $s$ with respect to $\leq$ and that supremum is contained in $P$.
	If $s=s_i$ for some index $i$, then $s_i$ is contained in $P\cap S'$ and is the supremum of the $s_i$ with respect to $\leq'$.
	We will show that $s$ is also the supremum of the $s_i$ with respect to $\leq'$ in the case that $s\neq s_i$ for all indices $i$.
	For that, let $i$ be some index.
	As $s_i < s$, and $s_i$ is contained in $S'$ and thus neither trivial nor degenerate, also $s$ is neither trivial nor degenerate.
	Furthermore, $s$ is contained in $P$ and hence cannot be cotrivial.
	Thus $s$ is contained in $S'$ and hence $s\in P\cap S'$.
	In order to show that $s_i\leq' s$, let $j$ be another index such that $s_i<' s_j$.
	Then $s_i<s_j<s$, and as $s_j$ is not trivial, $s_i\neq s^*$.
	Thus $s_i\leq' s$.
	Let $s'$ be another upper bound of the $s_i$ with respect to $\leq'$.
	Then $s_i<s<s'$, and as $s_i$ is not trivial, $s^*\neq s'$.
	Hence $s\leq' s'$, and $s$ is indeed the supremum of the $s_i$ with respect to $\leq'$.

	We will now show that if $S$ is orderly with respect to $\mathcal{P}$, then $S'$ is orderly with respect to $\mathcal{P}'$.
	So let $s$ and $t$ be elements of $S'$ and let $P$ and $Q$ be elements of $\mathcal{P}'$ such that $P\cap S'$ contains both $s$ and $t$ and $Q$ contains both $s^*$ and $t^*$.
	We will only show that $s\vee' t$ is defined, the fact that $s^*\vee' t^*$ is also defined then follows from swapping the roles of $P$ and $Q$.
	As $S$ is orderly with respect to $\mathcal{P}$, $s\vee t$ is defined and thus contained in $P$.
	In particular $s\vee t$ is not co-trivial.
	Consider the case that $s\vee t = s$.
	In that case $t\leq s$.
	As both $s$ and $t$ are contained in $P$ and are non-degenerate, $t\neq s^*$ and thus $t\leq' s$.
	Hence $s\vee' t$ is defined and we are done.
	Similarly if $s \vee t = t$ then we are done.
	So we may assume that $s < s\vee t$ and $t < s \vee t$.
	As $s$ is not trivial, $s\vee t$ is neither degenerate nor trivial, and thus $s\vee t$ is contained in $P \cap S'$.
	Again, as $s$ and $s\vee t$ are both contained in $P$ and non-degenerate, $s\neq (s\vee t)^*$ and thus $s\leq ' s\vee t$.
	Similarly $t \leq' s\vee t$.
	Finally, let $p$ be another upper bound of $s$ and $t$ in $\leq'$ with $s\vee t\neq p$.
	Then $s< s\vee t <p$, so if $s\vee t = p^*$ then $s$ is trivial.
	But $s$ is not trivial, and thus $s\vee t \neq p^*$ so $s\vee t \leq' p$.
	Hence $s\vee t$ is the supremum of $s$ and $t$ with respect to $\leq'$ and so $s\vee' t$ is defined.
\end{proof}

So we have shown that we can make a separation system with corner map and a set of profiles $\mathcal{P}$ regular.
Now we have to show that if we find a tree set of the regularization that distinguishes $\mathcal{P}'$, then we can translate it back to a nested subset of the original separation system distinguishing $\mathcal{P}$.

\begin{lem}\label{TreesetinRegularization}
	Let $(S,\leq, \mathord{^*}, \vee)$ be a separation system with a corner map and let $\mathcal{P}$ be a set of profiles of that separation system.
	Denote the regularization of the essential core of the separation system by $(S', \leq', \mathord{^*}', \vee')$.
	If $T$ is a tree set contained in $S'$ distinguishing $\{P\cap S'\colon P\in \mathcal{P}\}$, then $T$ is also a tree set of $S$ and distinguishes $\mathcal{P}$.
\end{lem}
\begin{proof}
	Let $T$ be a tree set contained in $S'$ that distinguishes $\{P\cap S'\colon P\in \mathcal{P}\}$.
	Then $T$ is also a nested subset of $S$ and a tree set of $\leq$.
	Every profile contains every trivial and degenerate separation, so every separation in $S$ distinguishing two profiles in $\mathcal{P}$ is also contained in $S'$.
	Hence, if $P$ and $Q$ are distinct elements of $\mathcal{P}'$, then $P\cap S'$ and $Q\cap S'$ are distinct profiles of $S'$ and thus $P\cap S'$ and $Q\cap S'$ are distinguished by $T$.
	So $T$ distinguishes all elements of $\mathcal{P}$.
\end{proof}

Let us now summarise the results of this section so far.

\begin{lem}\label{regularizationsum}
	Let $(S,\leq,\mathord{^*}, \vee)$ be a separation system with corner map, and let $\mathcal{P}$ be a set of profiles of this separation system.
	Then there is a regular separation system with corner map $(S',\leq'\mathord{^*}', \vee')$ and set of profiles $\mathcal{P}'$ such that every tree set distinguishing $\mathcal{P}'$ in $S'$ also distinguishes $\mathcal{P}$ in $S$.
	Moreover, if $S$ is orderly with respect to $\mathcal{P}$ then $S'$ is orderly with respect to $\mathcal{P}'$ and if the elements of $\mathcal{P}$ are closed with respect to $\leq$ then the elements of $\mathcal{P}'$ are closed with respect to $\leq'$.
\end{lem}

So now we can show that in \cref{abstracttheorem} the regularity condition can be dropped.
Note that in the proof, the only properties of the universe surrounding $S$ that is used are encoded by the properties of the induced corner map.
Therefore, \cref{abstracttheorem} also holds (with the same proof) if $S$ is not a subsystem of a universe but instead has a corner map.

So in particular we can prove the following generalisation of \cref{abstracttheorem}.

\begin{cor}
	Let $(S,\leq,\mathord{^*}, \vee)$ be a separation system with a corner map that is orderly with respect to a set $\mathcal{P}$ of closed profiles. Then there is a tree set $T$ with the following properties:
	\begin{enumerate}
		\item{Any two different elements $\mathcal{P}$ are distinguished by some element of $T$.}
		\item{Any element of $T$ distinguishes some elements of $\mathcal{P}$.}
		\item{In the set $P \cap T$ for every $P \in \mathcal{P}$ every separation lies below some maximal separation.}
	\end{enumerate}
\end{cor}
\begin{proof}
	By \cref{regularization}, the regularization $(S', \leq', \mathord{^*}', \vee')$ is orderly with respect to the set of profiles $\mathcal{P}' := \{P \cap S' \colon P \in \mathcal{P}\}$ and $\mathcal{P}'$ is closed in $S'$.
	We can now obtain by \cref{abstracttheorem} a tree set $T$ in $S'$ distinguishing the elements of $\mathcal{P}'$, which is a subset of $S$ as required in this Corollary by \cref{TreesetinRegularization}.
\end{proof}

\section{A tangle-tree theorem for submodular separation systems}\label{sec:ttt}
In this section, we will prove the following theorem.
Recall that a profile of a submodular universe is defined to be a $k$-profile of the universe for some $k\in \mathbb{N}$.
\begin{thm}[main theorem]\label{main}
	Let $U$ be a submodular universe and let $\mathcal{P}$ be a collection of distinguishable regular robust closed profiles. Then there is a tree set $T$ that efficiently distinguishes $\mathcal{P}$ such that every $t\in T$ efficiently distinguishes two profiles in $\mathcal{P}$.
\end{thm}

\begin{rem}
	The tree set $T$ constructed in the proof of \cref{main} additionally has the property that for every $P\in \Pcal$, every element of $T\cap P$ is less than or equal to a maximal element of $T\cap P$.
\end{rem}

We are going to recursively construct tree sets $T_k$ whose union then efficiently distinguishes $\mathcal{P}$.
To do this, let $\mathcal{P}_k$ be the set of profiles of order at most $k$ that are induced by elements of $\mathcal{P}$.
For the construction of a tree set $T_{k+1}$ distinguishing the elements of $\mathcal{P}_{k+1}$ we will employ the following proof strategy in two steps, that is common in proofs of tree-of-tangles-theorems, for example in \cite{ProfilesNew}.
First, for a $k$-profile $P\in \mathcal{P}_k$, \cref{abstracttheorem} is applied to the set of all $k+1$-profiles in $\mathcal{P}_{k+1}$ whose induced $k$-profile is $P$ and to a carefully chosen subuniverse of $U$ to obtain a tree set $T_P$.
Second, it will be shown that $T_k$ together with all the tree sets $T_P$ is a tree set that efficiently distinguishes $\mathcal{P}_{k+1}$.

More precisely, we are going to show by induction that for every $k\in \mathbb{N}$ there is a tree set $T_k$ with the following properties:
\begin{itemize}
	\item $T_k$ is a subset of $S_k$, and if $T_{k-1}$ exists then $T_{k-1} \subseteq T_k$.
	\item Every element of $T_k$ distinguishes two profiles in $\mathcal{P}_k$ efficiently.
	\item $T_k$ distinguishes $\mathcal{P}_k$ efficiently.
	\item For every $Q\in \mathcal{P}_k$, every element of $Q\cap T_k$ is less than or equal to a maximal element of $Q\cap T_k$.
\end{itemize}

As the only $0$-profile is the empty set we define $T_0$ to be the empty set.
Assume that $T_k$ is already defined and we now want to define $T_{k+1}$.
For each $k$-profile $P\in\mathcal{P}_k$ let \QP\ be the set of all $(k+1)$-profiles in $\mathcal{P}_{k+1}$ whose induced $k$-profile is $P$, $N_P$ the set of maximal elements of $P\cap T_k$ and $U_P$ the set of all separations in $U$ towards which all separations of $N_P$ point.
Note that $U_P$ is a universe and closed under taking infima and suprema of chains of bounded order, and that as a result the set of $k$-separations of $U_P$ is structurally submodular and closed under taking suprema of chains of bounded order.
Then every profile in \QP\ induces a closed $(k+1)$-profile of $U_P$.

We will want to apply the results of the previous section, in particular \cref{abstracttheorem}, to $U_P$ and the profiles of $U_P$ induced by $\QP$.
In order to do so, we need to show that the separation system $S_k$ of $U_P$ is orderly with respect to $\QP$.
That follows from the following, slightly more general statement.

\begin{lem}\label{orderly}
	Let $U$ be a submodular universe, $k$ a non-negative integer and $\mathcal{P}$ a set of $k+1$-profiles that all have the same induced $k$-profile.
	Then the set of $k$-separations of $U$ is orderly with respect to $\mathcal{P}$.
\end{lem}
\begin{proof}
	Assume $s$ and $t$ are $k$-separations of $U$ and $P$ and $Q$ are elements of $\mathcal{P}$ such that $\{s,t\}\subseteq P$ and $\{s^*,t^*\}\subseteq Q$.
	By submodularity, one of $s\vee t$ and $s\wedge t$ is also a $k$-separation, assume without loss of generality that $s\vee t$ is a $k$-separation.
	Then $s\vee t$ is also contained in $P$.

	First consider the case that $s\vee t$ is not contained in $Q$.
	Then $s\vee t$ distinguishes $P$ and $Q$, and thus has order exactly $k$.
	Hence by submodularity also $s\wedge t$ is a $k$-separation and we are done.

	So consider the case that $s\vee t\in Q$.
	As a separation system with a degenerate separation does not have a profile, $s$ is not contained in $Q$.
	Thus by consistency $s\leq s\vee t\in Q$ implies $s=(s\vee t)^*$.
	Similarly $t=(s\vee t)^*$, so $s=t$ and the lemma holds.
\end{proof}

Now we will show that indeed all profiles in $\QP$ can be distinguished by separations of $U_P$.

\begin{lem}\label{mayrestricttoU_P}
	All profiles in \QP\ are distinguished by separations of $U_P$ of order $k$.
\end{lem}
\begin{proof}
	Let $f_{\QP}:U \rightarrow 2^{\QP}$ be defined as \cref{sec:equivclasses}.
	It suffices to show that if $r$ is a separation such that $f_{\QP}(r)$ is neither $\emptyset$ nor $\QP$, then there is an element of $U_P$ that has the same image under $f_{\QP}$ as $r$.
	Let $X$ be the set of separations whose image under $f_{\QP}$ is $f_{\QP}(r)$.
	Several times in this proof we will use the fact that $f_{\QP}$ preserves infima and suprema, and in particular any two elements of $X$ have an infimum and a supremum.

	If $N_P$ is empty, then $r$ itself is a separation of $U_P$ and we are done, so assume otherwise.
	By \cref{biggestequivalent}, $X$ has a smallest element $s$.
	Let $N'$ be the set of all elements $n$ of $N_P$ such that $n\leq s^*$.
	Define $X'$ to contain all elements $x\in X$ with $n\leq x^*$ for all $n\in N'$.
	Just as in the proof of \cref{biggestequivalent}, the fact that the profiles in $\QP$ are closed and Zorn's Lemma imply that $X'$ has a maximal element $t$.
	See \cref{fig:nestedwithNP} for an illustration.

	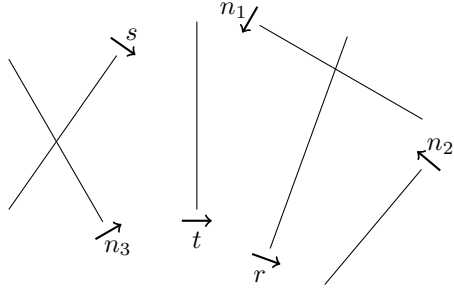
\begin{figure}
		\centering
		\begin{tikzpicture}
			\draw (-2cm,1cm) -- ++(-60:2.5cm) pic [rotate = 30]{separrow = $n_3$};
			\draw (-2cm, -1cm) -- ++(55:2.5cm) pic [rotate = 145] {separrow inv= $s$};
			\draw (0.5cm, 1.5cm) -- ++(-90:2.5cm) pic {separrow = $t$};
			\draw (3.5cm, .2cm) -- ++(150:2.5cm) pic [rotate = 240] {separrow = $n_1$};
			\draw (2.2cm, -2cm) -- ++(50:2cm) pic [rotate = 140] {separrow = $n_2$};
			\draw (2.5cm, 1.3cm) -- ++(-110:3cm) pic [rotate = -20] {separrow = $r$};
		\end{tikzpicture}\caption{The separation $s$ is less than the separations $n_1^*$ and $n_2^*$, which is equivalent to $n_1\leq s^*$ and $n_2\leq s^*$, and the separation $t$ is additionally bigger than $n_3$.}\label{fig:nestedwithNP}
	\end{figure}

	We will show that all elements of $N_P$ point towards $t$.
	In order to do so let $n$ be an element of $N_P$.
	If there is $q\in X$ such that $q\leq n^*$, then $s\wedge q\in X$.
	By minimality of $s$ this implies $s=s\wedge q$, so $s\leq q\leq n^*$ and thus $n\in N'$.
	Hence $n$ points towards~$t$.
	The other case is where there is no such separation $q$, in particular $t\wedge n^*$ is not a candidate for $q$.
	If $t\wedge n^*$ is a $\leq k$-separation, then it is contained in $X$, because $n$ is contained in all elements of $\QP$.
	So the fact that $t \wedge n^*$ is not contained in $X$ implies that it is not a $\leq k$-separation.
	Because $\mathcal{P}$ is robust and consistent, this implies that $t\vee n = (t^* \wedge n^*)^*$ efficiently distinguishes the same elements of $\QP$ that $t$ distinguishes.
	Then $t$ is a $k$-separation and thus contained in $X$, and all $n'\in N'$ point towards $t\vee n$, so by maximality of $t$ in $X'$ we have $t\vee n=t$ and hence $n\leq t$.
	So also in this case $n$ points towards $t$.
\end{proof}

So in order to distinguish the elements of $\mathcal{Q}_P$, it suffices to distinguish their intersections with $U_P$.
We will do that by applying \cref{abstracttheorem} to the $k$-separations of $U_P$ and the profiles induced by $\mathcal{Q}_P$.
Call the obtained tree set $T_P$.
Let $T_{k+1}$ be the union of $T_k$ and all tree sets $T_P$ where $P$ is a $k$-profile contained in $\mathcal{P}_k$.

\begin{lem}
	$T_{k+1}$ is a tree set that distinguishes $\mathcal{P}_{k+1}$ efficiently, and in which every separation distinguishes some elements of $\mathcal{P}_{k+1}$.
\end{lem}
\begin{proof}
	Let $P$ be a $k$-profile in $\mathcal{P}_k$.
	As for every separation $n$ in $P\cap T_k$ there is an element $n'$ of $N_P$ such that $n\leq n'$, $T_k$ is nested with every separation in $U_P$ and thus with $T_P$.
	Also, if $P$ and $P'$ are distinct $k$-profiles in $\mathcal{P}_k$, then they are distinguished by some separation $n$ of $T_k$, so they are also distinguished by some separation of $N_P$ which then witnesses that every separation in $T_{P'}$ is nested with every separation in $T_P$.
	So $T_k$ is nested.
	Also by construction every element of $T_{k+1}$ distinguishes two elements of $\mathcal{P}_{k+1}$ and thus is neither small nor cosmall.
	Hence every element of $T_{k+1}$ is neither trivial nor co-trivial nor degenerate.

	In order to show that every two elements of $\mathcal{P}_{k+1}$ are distinguished efficiently, let $P$ and $Q$ be two such elements that can be distinguished.
	If $P$ and $Q$ can be distinguished by a separation of order at most $k-1$, then they induce distinguishable elements of $\mathcal{P}_k$ which are thus efficiently distinguished by some separation $t$ of $T_k$.
	Then $t$ also efficiently distinguishes $P$ and $Q$.
	So we are left with the case that $P$ and $Q$ cannot be distinguished by a separation of order less than $k$, which implies that they are both $(k+1)$-profiles and induce the same $k$-profile $P'$.
	But then $P$ and $Q$ are distinguished by a separation in $T_{P'}$, and that separation distinguishes them efficiently.
\end{proof}

So in order to complete the proof of \cref{main}, we only have to prove the following statement.

\begin{lem}
	For every $Q\in \mathcal{P}_{k+1}$, every element of $Q \cap T_{k+1}$ is less than or equal to a maximal element of $Q \cap T_{k+1}$.
\end{lem}
\begin{proof}
	Let $Q\in \mathcal{P}_{k+1}$ and let $s\in Q \cap T_{k+1}$.
	By induction it suffices to consider the case that $Q$ has order $k+1$.
	Denote the induced $k$-profile of $Q$ by $Q'$.
	If $s$ is not contained in $T_{Q'}$, then it is less than or equal to an element of $N_{Q'}$, and every element of $N_{Q'}$ is either maximal in $T_{k+1}\cap Q$ or less than a separation in $Q \cap T_{Q'}$.
	So it suffices to consider the case that $s\in T_{Q'}$.
	But in this case by \cref{abstracttheorem} $s$ is less than or equal to a maximal element of $T_{Q'}\cap Q$ which then also is a maximal element of $Q \cap T_{k+1}$.
\end{proof}

\begin{proof}[Proof of \cref{main}]
	Let $T$ be the union of all the $T_k$ as defined in this section.
	Then $T$ is a tree set which distinguishes all distinguishable profiles in $\mathcal{P}$ efficiently, and every separation in it distinguishes two elements of $\mathcal{P}$ efficiently.
\end{proof}

\printbibliography

\end{document}